\documentclass[a4paper,12pt]{article}
\usepackage{graphicx}
\usepackage{a4wide}
\usepackage[english]{babel}
\usepackage{amsfonts}
\usepackage{amsmath}
\usepackage{amssymb}
\usepackage{dsfont}

\allowdisplaybreaks
\usepackage{amsmath,amssymb}
\usepackage[latin2]{inputenc}
\usepackage{amsthm}
\usepackage{amssymb}
\usepackage[T1]{fontenc}

\newtheorem{lem}{Lemma}
\newtheorem{thm}{Theorem}

\theoremstyle{remark}

\newcommand {\E} {\mathbb{E}}
\DeclareMathOperator {\var}{Var_{[0,1]}}

\numberwithin{equation}{section} \numberwithin{theorem}{section}
\numberwithin{corollary}{section}

\begin{document}

\begin{center}
   {\bf {\Large On permutations of Hardy-Littlewood-P\'olya sequences}}

\end{center}
\vspace{.3cm}
\renewcommand{\baselinestretch}{1.05}
\normalsize

\begin{center}
   {\sc
    Christoph Aistleitner
    \footnote{Department of Mathematics A, Graz University of Technology,
    Steyrergasse 30, A--8010 Graz, Austria,
    email: {\tt aistleitner@finanz.math.tugraz.at} \ Research supported by FWF grant S9603-N23 and a MOEL scholarship of the \"Osterreichische
    Forschungsgemeinschaft.} \qquad
    Istv\'an Berkes
    \footnote{Institute of Statistics, Graz University of Technology,
    M\"unzgrabenstrasse  11, A--8010 Graz, Austria,
    email: {\tt berkes@tugraz.at} \ Research supported by the FWF Doctoral Program on Discrete Mathematics (FWF DK W1230-N13), FWF grant S9603-N23 and OTKA grants K 61052 and K 67986.} \qquad
    Robert  F.\ Tichy
    \footnote{Department of Mathematics A, Graz University of Technology,
    Steyrergasse 30, A--8010 Graz, Austria,
    email: {\tt tichy@tugraz.at} Research supported by the FWF Doctoral Program on Discrete Mathematics (FWF DK W1230-N13) and FWF grant S9603-N23.}}
\end{center}

\bigskip
\begin{center}
{\it Dedicated to the memory of Walter Philipp}
\end{center}

\vskip1.5cm

{\noindent{\bf Abstract:}\ \ Let ${\cal H}=(q_1, \ldots q_r)$ be a
finite set of coprime integers and let $n_1, n_2, \ldots$ denote
the multiplicative semigroup generated by $\cal H$ and arranged in
increasing order. The distribution of such sequences has been
studied intensively in number theory and they have remarkable
probabilistic and ergodic properties. For example, the asymptotic
properties of  the sequence $\{n_kx\}$ are very similar to those
of independent, identically distributed random variables; here
$\{\cdot \}$ denotes fractional part. However, the behavior of
this sequence depends sensitively on the generating elements of
$(n_k)$ and the combination of probabilistic and number-theoretic
effects results in a unique, highly interesting asymptotic
behavior, see e.g.\ \cite{ft}, \cite{ft3}. In particular, the
properties of $\{n_kx\}$ are not permutation invariant, in
contrast to i.i.d.\ behavior. The purpose of this paper is to show
that $\{n_kx\}$ satisfies  a strong independence property
("interlaced mixing"), enabling one to determine the precise
asymptotic behavior of permuted sums $S_N (\sigma)= \sum_{k=1}^N
f(n_{\sigma(k)} x)$. As we will see, the behavior of $S_N(\sigma)$
still follows that of sums of independent random variables, but
its growth speed (depending on $\sigma$) is given by the classical
G\'al function of Diophantine approximation theory.
Some examples describing the class of possible growth functions
are given.



\vskip2cm

\newpage

\section{Introduction}

\bigskip
Let $q_1, \ldots, q_r$ be a fixed set of coprime integers and let
$(n_k)$ be the set of numbers $q_1^{\alpha_1}\cdots q_
r^{\alpha_r}$, $\alpha_i\ge 0$ integers, arranged in increasing
order. Such sequences are called (sometimes)
Hardy-Littlewood-P\'olya sequences and their distribution has been
investigated extensively in number theory. Thue \cite{thue} showed
that $n_{k+1}-n_k\to\infty$ and this result was improved gradually
until Tijdeman \cite{tij} proved that
$$n_{k+1}-n_k\ge \frac{n_k}{(\log n_k)^\alpha}$$
for some $\alpha>0$, i.e.\ the growth of $(n_k)$ is almost
exponential.  Except the value of the constant $\alpha$, this
result is best possible. Hardy-Littlewood-P\'olya sequences also
have remarkable probabilistic and ergodic properties.  In his
celebrated paper on the Khinchin conjecture, Marstrand  \cite{mar}
proved that if $f$ is a bounded measurable function with period 1,
then
$$
\lim_{N\to\infty} \frac{1}{N} \sum_{k=1}^N f(n_kx)=\int_0^1 f(t)
dt\quad \text{a.e.}
$$
and Nair \cite{na} showed (cf.\ Baker \cite{ba}) that this remains
valid if instead of boundedness of $f$ we assume only $f\in L^1(0,
1)$. Letting $\{ \cdot \}$ denote fractional part, it follows that
$\{n_kx\}$ is not only uniformly distributed mod 1 for almost all
$x$ in the sense of Weyl \cite{we}, but satisfies the  "strong
uniform distribution" property of Khinchin \cite{kh}.  Letting
$$
D_N=D_N( x_1, \ldots, x_N):= \sup_{0 \leq a < b < 1} \left|
\frac{1}{N} \#\{k\le N: a\le x_k <b\} - (b-a) \right|
$$
denote the discrepancy of a sequence $(x_k)_{1\le k \le N}$ in
$(0, 1)$, Philipp \cite{ph1994} proved, verifying a conjecture of
R.C.\ Baker, that
\begin{equation}\label{ph1994lil}
1/8\le  \limsup_{N\to\infty}  \sqrt{\frac{N}{2\log\log N}}
D_N(\{n_kx\}_{1\le k \le N}) \le C \qquad \text{a.e.},
\end{equation}
with a constant $C$ depending on the generating elements of
$(n_k)$, establishing the law of the iterated logarithm for the
discrepancies of $\{n_kx\}$. Note that if $(\xi_k)$ is a sequence
of independent random variables with uniform distribution over
$(0, 1)$, then
\begin{equation}\label{cslil}
 \limsup_{N\to\infty}  \sqrt{\frac{N}{2\log\log N}}
D_N(\xi_1, \ldots, \xi_N) =\frac{1}{2} \qquad
\end{equation}
with probability one by the Chung-Smirnov LIL (see e.g.\
\cite{sw}, p.\ 504). A comparison of (\ref{ph1994lil}) and
(\ref{cslil}) shows that the sequence $\{n_kx\}$ behaves like a
sequence of independent random variables. In the same direction,
Fukuyama and Petit \cite{fupe} showed that under mild assumptions
on the periodic function $f$, $\sum_{k\le N} f(n_kx)$ obeys the
central limit theorem, another remarkable  probabilistic property
of Hardy-Littlewood-P\'olya sequences. Surprisingly, however, the
limsup in (\ref{ph1994lil}) is different from the constant $1/2$
in (\ref{cslil}) and, as Fukuyama \cite{ft} and Fukuyama and
Nakata \cite{ft3}  showed, it depends sensitively on the
generating elements $q_1, \ldots, q_r$. For example, for
$n_k=a^k$, $a\ge 2$ the limsup $\Sigma_a$  in (\ref{ph1994lil})
equals

\begin{eqnarray}
\Sigma_a & = & \sqrt{42}/9  \phantom{999999999999} \textrm{if} ~ a =2 \nonumber\\
\Sigma_a & = & \frac{\sqrt{(a+1)a(a-2)}}{2\sqrt{(a-1)^3}} \quad
\textrm{if} ~ a \geq 4 ~
\textrm{is an even integer},\nonumber\\
\Sigma_a & = & \frac{\sqrt{a+1}}{2\sqrt{a-1}} \phantom{9999999999}
\textrm{if} ~ a \geq 3 ~ \textrm{is an odd integer} \nonumber,
\end{eqnarray}
and if all the generating elements $q_i$ of $(n_k)$ are odd, then
the limsup in (\ref{ph1994lil}) equals
$$ \frac{1}{2} \left(\prod_{i=1}^r \frac{q_i+1}{q_i-1}\right)^{1/2}.$$
Even more surprisingly, Fukuyama \cite{ft2} showed that the limsup
$\Sigma$ in (\ref{ph1994lil}) is not permutation-invariant:
changing the order of the $(n_k)$ generally changes the value of
$\Sigma$. This is quite unexpected, since $\{n_kx\}$ are
identically distributed in the sense of probability theory and the
asymptotic properties of i.i.d.\ random variables are  permutation
invariant. The purpose of this paper is to give a  detailed study
of the structure of $\{n_kx\}$ in order to explain the role of
arithmetic effects and the above surprising deviations from
i.i.d.\ behavior. Specifically, we will establish an  "interlaced"
mixing condition for normed sums of $\{n_kx\}$, expressed by
Lemmas \ref{lem:3}  and \ref{lem:5}, implying that the sequence
$\{n_kx\}$ has mixing properties after any permutation of its
terms. This property is considerably stronger than usual mixing
properties of lacunary sequences, which are always directed, i.e.\
are valid only in the "natural" order of elements.
In particular, we will see that for any permutation $\sigma:
{\mathbb N}\to {\mathbb N}$ of the positive integers, $\sum_{k\le
N} f(n_{\sigma(k)}x)$  still behaves like sums of independent
random variables and the observed pathological properties of these
sums are due to the unusual behavior of their $L^2$ norms which,
as we will see, is a purely number theoretic effect. For example,
in the case $f(x)=\{x\}$ the growth speed of the above sums is
determined by $G(n_{\sigma(1)}, \ldots, n_{\sigma(N)})$, where
\begin{equation}\label{gal}
G(m_1, \ldots, m_N)=\sum_{1\le i\le j \le N} \frac{ (m_i,
m_j)}{[m_i, m_j]}
\end{equation}
is the G\'al function in Diophantine approximation theory; here
$(a,b)$ and $[a,b]$ denote the greatest common divisor, resp.\
least common multiple of $a$ and $b$. While this function is
completely explicit, the computation of its precise asymptotics
for a specific permutation $\sigma$ is a challenging problem  and
we will illustrate the situation only by a few examples.

As noted, the basic structural information on $\{n_kx\}$ is given
by Lemmas \ref{lem:3} and \ref{lem:5}, which are rather technical.
The following result, which is a simple consequence of them,
describes the situation more explicitly.

\begin{thm}\label{th1}
Let $f:{\mathbb R}\to{\mathbb R}$ be a measurable function
satisfying the condition
\begin{equation} \label{fcond}
f(x+1)=f(x), \quad \int_0^1 f(x)~dx=0, \quad \var f < + \infty
\end{equation}
and let $\sigma:{\mathbb N}\to {\mathbb N}$ be a permutation of
$\mathbb N$. Assume that
\begin{equation}\label{anm}
A_{N, M}^2:= \int_0^1 \left(\sum_{k=M+1}^{M+N} f(n_{\sigma
(k)}x)\right)^2 dx \ge CN, \qquad N\ge N_0, \ M\ge 1
\end{equation}
for some constant $C>0$. Then letting $A_N=A_{N, 0}$ we have
\begin{equation}\label{fclt}
 A_N^{-1}\sum_{k=1}^N f(n_{\sigma (k)}x)\to_d N(0, 1)
\end{equation}
and
\begin{equation}\label{lilperm}
\limsup_{N\to \infty} \frac{1}{(2A_N^2\log\log A_N^2)^{1/2}}
\sum_{k=1}^N
f(n_{\sigma (k)} x)=1 \quad\textup{a.e.}\\
\end{equation}
\end{thm}

\medskip
As the example $f(x)=\cos 2\pi x -\cos 4\pi x$, $n_k=2^k$ shows,
assumption (\ref{anm}) cannot be omitted in Theorem \ref{th1}. It
is satisfied, e.g., if all Fourier coefficients of $f$ are
nonnegative.

Theorem \ref{th1} shows that the growth speed of $\sum_{k=1}^N
f(n_{\sigma (k)}x)$ is determined by the quantity
$$A_N^2=A_n^2(\sigma)=\int_0^2 \left(\sum_{k=1}^N
f(n_{\sigma (k)}x)\right)^2 dx.
$$
In the harmonic case $f(x)=\sin 2\pi x$ we have $A_N
(\sigma)=\sqrt{N/2}$ for any $\sigma$ and thus the partial sum
behavior is permutation-invariant. For trigonometric polynomials
$f$ containing at least two terms the situation is different: for
example, in the case $f(x)=\cos 2\pi x +\cos 4\pi x$ the limits
$\lim_{N\to\infty} A_N (\sigma)/\sqrt{N}$ for all permutations
$\sigma$ fill an interval. In the case $f(x)= \{x\}-1/2$ we have,
by a well known identity of Landau (see \cite{la}, p.\ 170)
$$ \int_0^1 f(ax) f(bx) dx =\frac{1}{12}(a, b)/[a, b]$$
 Hence in this case
$$
A_N^2 =\frac{1}{12} G(n_{\sigma(1)}, \ldots, n_{\sigma(N)})
$$
where $G$ is the G\'al function defined by $(\ref{gal})$. The
function $G$ plays an important role in the metric theory of
Diophantine approximation and it is generally very difficult to
estimate; see the profound paper of G\'al \cite{ga} for more
information on this point. Clearly, $G(n_{\sigma(1)}, \ldots,
n_{\sigma(N)})\ge N$ and from the proof of Lemma 2.2 of Philipp
\cite{ph1994} it is easily seen that
$$G(n_{\sigma(1)}, \ldots, n_{\sigma(N)})\ll N.$$
Here, and in the sequel, $\ll$ means the same as the $O$ notation.
In the case of the identity permutation $\sigma$ the value of
$\lim_{N\to\infty} N^{-1} G(n_1, \ldots, n_N)$ was computed by
Fukuyama and Nakata \cite{ft3}, but to determine the precise
asumptotics  of $G(n_{\sigma(1)}, \ldots, n_{\sigma(N)})$ for
general $\sigma$ seems to be a very difficult problem. Again, in
Section 3 we will see that in the case of $n_k=2^k$ the class of
limits $\lim_{N\to\infty} N^{-1} G(n_{\sigma(1)}, \ldots,
n_{\sigma(N)})$ for all $\sigma$ fills an interval.


\bigskip\noindent
{\bf Corollary. } {\it Let $f:{\mathbb R}\to{\mathbb R}$ be a
measurable function satisfying (\ref{fcond}) and assume that the
Fourier coefficients of $f$ are nonnegative.
Let $\sigma$ be a permutation of $\mathbb N$. Then $
N^{-1/2}\sum_{k=1}^N f(n_{\sigma (k)}x)$ has a nondegenerate limit
distribution iff
\begin{equation}\label{gamma1}
\gamma^2= \lim_{N\to\infty} N^{-1}\int_0^1 \left(\sum_{k=1}^n
f(n_{\sigma (k)}x)\right)^2 dx >0
\end{equation}
exists, and then
\begin{equation}\label{fclt2}
 N^{-1/2}\sum_{k=1}^N f(n_{\sigma (k)}x)\to_d N(0, \gamma^2).
\end{equation}
Also, if condition (\ref{gamma1}) is satisfied, then
\begin{equation}\label{lilperm2}
\limsup_{N\to \infty} \frac{1}{\sqrt{2N\log\log N}} \sum_{k=1}^N
f(n_{\sigma (k)} x)=\gamma \quad\textup{a.e.}\\
\end{equation}
}

As mentioned, for the original, unpermuted sequence $(n_k)$, the
value of $\gamma=\gamma_f$ in (\ref{gamma1}) was computed in
\cite{ft3}. Given an $f$ satisfying condition (\ref{fcond}), let
$\Gamma _f$ denote the set of limiting variances in (\ref{gamma1})
belonging to all permutations $\sigma$. (Note that the limit does
not always exist.) Despite the simple description of $\Gamma_f$
above, it seems a difficult problem to determine this set
explicitly. In analogy with the theory of permuted function series
(see e.g.\ Nikishin \cite{ni}), it is natural to expect that
$\Gamma_f$ is always a (possibly degenerate) interval. In Section
3 we will prove that for $n_k=2^k$ and functions $f$ with
nonnegative Fourier coefficients, $\Gamma_f$ is identical with the
interval determined by $\|f\|^2$ and $\gamma_f^2$. For $f$ with
negative coefficients this is false, as an example in Section 3
will show.

\section{An interlaced mixing condition}

\bigskip
The crucial tool in proving Theorem 1 is a recent deep bound for
the number of solutions $(k_1, \ldots, k_p)$ of the Diophantine
equation
\begin{equation}\label{dio}
a_1 n_{k_1}+\ldots +a_p n_{k_p}=b.
\end{equation}
Call a solution of (\ref{dio}) nondegenerate if no subsum of the
sum on the left hand side equals 0. Amoroso and Viada \cite{am}
proved the following result, improving the quatitative subspace
theorem of Schmidt \cite{sch} (cf.\  also Evertse et al.\
\cite{ev}).

\begin{lem}\label{evertse} For any nonzero integers  $a_1, \ldots,
a_p, b$ the number of nondegenerate solutions of (\ref{dio}) is at
\ most $\exp(cp^6)$, where $c$ is a constant depending only on the
number of generators of $(n_k)$.
\end{lem}

For the rest of the paper, $C$ will denote positive constants,
possibly different at different places, depending (at most) on $f$
and $(n_k)$. Similarly, the constants implied by $O$ and by the
equivalent relation $\ll$ will depend (at most) on $f$ and
$(n_k)$.

Most results of this paper are probabilistic statements on the
sequence $\{f(n_kx), \, k=1, 2, \ldots\}$ and we will use
probabilistic terminology. The underlying probability space for
our sequence is the interval $[0, 1]$, equipped with Borel sets
and the Lebesgue measure; we will denote probability and
expectation in this space by $\mathbb P$ and $\mathbb E$.

Given any finite set $I$ of positive integers, set
$$
S_I = \sum_{k \in I} f(n_kx), \quad \sigma_I=
(\E S_I^2)^{1/2}.
$$
From Lemma 1 we deduce

\begin{lem} \label{lemma2}
Assume the conditions of Theorem 1 and let $I$ be a set of
positive integers with cardinality $N$. Then we have for any
integer $p\ge 3$
$$
\E S_I^p=
\begin{cases}
\frac{p!}{(p/2)!}2^{-p/2} \sigma_I^p+O(T_N)
 \quad &\text{if} \ p \ \text{is even} \\
O(T_N) \quad &\text{if} \ p \ \text{is odd}
\end{cases}
$$
where
\begin{equation*}
T_N=\exp (Cp^7) N^{(p-1)/2}(\log N)^p.
\end{equation*}
\end{lem}

\noindent {\bf Proof.} Let $C_p=\exp (cp^6)$ be the constant in
Lemma \ref{evertse}. We first note that
\begin{equation}\label{starstar}
\sigma_I\le K\|f\|^{1/2} |I|^{1/2},
\end{equation}
where $K$ is a constant depending  only on the generating elements
of $(n_k)$. This relation is implicit in the proof of Lemma 2.2 of
Philipp \cite{ph1994}. Next we observe that for any fixed $p\ge 3$
and any fixed nonzero coefficients $a_1, \ldots, a_p$, the number
of nondegenerate solutions of (\ref{dio}) such that $b=0$ and
$k_1, \ldots, k_p \in I$ is at most $C_{p-1}N$. Indeed, the number
of choices for $k_p$ is at most $N$, and thus taking $a_pn_{k_p}$
to the right hand side and applying Lemma 1, our claim follows.

Without loss of generality we may assume that $f$ is an even
function and that $\|f\|_\infty \le 1$, $\var f \leq 1$; the proof
in the general case is  the same. (Here, and in the sequel, $\|
\cdot \|_p$ denotes the $L_p$ norm; for $p=2$ we simply write
$\|\cdot \|$.) Let
\begin{equation*}
f \sim \sum_{j=1}^\infty a_j \cos 2 \pi j x
\end{equation*}
be the Fourier series of $f$. $\var f \leq 1$ implies (see Zygmund
\cite[p.\ 48]{zt})
\begin{equation} \label{fouriercoeffs}
|a_j| \leq j^{-1},
\end{equation}
and, writing
$$
g(x) = \sum_{j=1}^{N^{2p}} a_j \cos 2 \pi j x, \qquad r(x) = f(x)
- g(x),
$$
we have
$$
\|g\|_\infty \leq \var f + \|f\|_\infty \leq 2, \qquad
\|r\|_\infty \leq  \|f\|_\infty + \|g\|_\infty \leq 3.
$$
For any positive integer $n$, (\ref{fouriercoeffs}) yields
$$
\|(r(n x)\|^2 = \|r(x)\|^2 = \frac{1}{2} \sum_{j=N^{2p}+1}^\infty
a_j^2 \leq N^{-2p}.
$$
By Minkowski's inequality,
\begin{equation*}\label{min1}
\|S_I\|_p \leq \left\| \sum_{k\in I} g(n_k x) \right\|_p + \left\|
\sum_{k\in I} r(n_k x) \right\|_p,
\end{equation*}
and
\begin{equation}\label{min2}
\left\| \sum_{k\in I} r(n_k x) \right\|_p  \leq 3\sum_{k\in I}
\left\| r(n_k x)/3 \right\|_p  \leq 3 \sum_{k\in I} \left\| r(n_k
x)/3 \right\|^{2/p} \leq 3 \sum_{k\in I} N^{-2} \leq 3.
\end{equation}
By expanding and using elementary properties of the trigonometric
functions we get
\begin{align}\label{esnp}
&\E\left(\sum_{k\in I} g(n_kx)\right)^p\nonumber\\
&= 2^{-p} \sum_{1\le j_1, \ldots, j_p \le N^{2p}} a_{j_1}\cdots
a_{j_p}\sum_{k_1, \ldots, k_p\in I} ~\mathds{I} \{\pm j_1 n_{k_1}
\pm \ldots \pm j_p n_{k_p}=0\},
\end{align}
with all possibilities of the signs $\pm$ within the indicator
function. Assume that $j_1, \dots, j_p$ and the signs $\pm$ are
fixed, and consider a solution of  $\pm j_1 n_{k_1} \pm \ldots \pm
j_p n_{k_p}=0$. Then the set $\{1, 2, \ldots, p\}$ can be split
into disjoint sets $A_1, \ldots, A_l$ such that for each such set
$A$ we have $\sum_{i\in A} \pm j_i n_{k_i}=0$ and no further
subsums of these sums are equal to 0. By the monotonicity of $C_p$
and the remark at the beginning of the proof, for each $A$ with
$|A|\ge 3$ the number of solutions is $\le C_{|A|-1}N\le
C_{p-1}N$;
trivially, for $|A|=2$ the number of solutions is at most $N$.
Thus if $s_i=|A_i|$ \ ($1\le i \le p$) denotes the cardinality of
$A_i$, the number of solutions of $\pm j_1 n_{k_1} \pm \ldots \pm
j_p n_{k_p}=0$ admitting such a decomposition with fixed $A_1,
\ldots, A_l$ is at most
\begin{align*}
&\prod_{\{i: s_i\ge 3\}} C_{p-1}N  \prod_{\{i: s_i=2\}} N \le
(C_{p-1} N)^{\sum_{\{i: s_i\ge 3\}}1+ {\sum_{\{i: s_i=2\}}1}} \\
&\le (C_{p-1} N)^{\frac{1}{3}\sum_{\{i: s_i\ge 3\}}s_i+
\frac{1}{2}{\sum_{\{i: s_i=2\}}s_i}}
= (C_{p-1}N)^{\frac{1}{3}\sum_{\{i: s_i\ge 3\}}s_i +\frac{1}{2}(p-
\sum_{\{i: s_i\ge
3\}}s_i)}\\
&=(C_{p-1}N)^{\frac{p}{2}-\frac{1}{6}\sum_{\{i: s_i\ge 3\}}s_i}.
\end{align*}
If there is at least one $i$ with $s_i\ge 3$, then the last
exponent is at most $(p-1)/2$ and since the number of partitions
of the set $\{1, \ldots, p\}$ into disjoint subsets is at most
$p!\, 2^p$, we see that the number of solutions of $\pm j_1
n_{k_1} \pm \ldots \pm j_p n_{k_p}=0$ where at least one of the
sets $A_i$ has cardinality $\ge 3$ is at most $p!\,
2^p(C_{p-1}N)^{(p-1)/2}$. If $p$ is odd, there are no other
solutions and thus using (\ref{fouriercoeffs}) the inner sum in
(\ref{esnp}) is at most $p!\, 2^p(C_{p-1}N)^{(p-1)/2}$ and
consequently, taking into account the $2^p$ choices for the signs
$\pm 1$,
\begin{align*}\label{esnpodd}
&\left|\E\left(\sum_{k\in I} g(n_kx)\right)^p\right| \nonumber\\
&\le p!\, 2^p(C_{p-1}N)^{(p-1)/2} 2^p \sum_{1\le j_1, \ldots, j_p
\le N^{2p}} |a_{j_1}\cdots a_{j_p}| \ll \exp(Cp^7) N^{(p-1)/2}
(\log N)^p.
\end{align*}
If $p$ is even, there are also solutions where each $A$ has
cardinality 2. Clearly, the contribution of the terms in
(\ref{esnp}) where $A_1=\{1, 2\}, A_2=\{3, 4\}, \ldots$ is
\begin{align}
&\left( \frac{1}{4}\sum_{1\le i,j \le N^{2p}} \sum_{k, \ell \in I}
a_i a_j \mathds{I}\{ \pm i n_k \pm j n_\ell=0\}\right)^{p/2}  =
\left( \E \left(\sum_{k\in I} g(n_k x) \right)^2 \right)^{p/2} \nonumber \\
&=\left\| \sum_{k\in I} g(n_kx)\right\|^{p}= \left\|S_I-
\sum_{k\in I} r(n_kx)\right\|^{p}=\left( \sigma_I +O(1)
\right)^{p} \nonumber \\
&=
\sigma_I^{p}+p \left(\sigma_I +O(1)\right)^{p-1}O(1)\nonumber \\
&=\sigma_I^{p}
+O\left(p 2^{p-2}\right) \left( \sigma_I^{p-1}+O(1)^{p-1}\right)\nonumber\\
&= \sigma_I^{p} + O(2^{p^2}) N^{(p-1)/2},
\end{align}
using the mean value theorem and the relation
\begin{equation}\label{conv}
\left(\sum_{j=1}^m x_j\right)^\alpha\le \max(1, m^{\alpha-1})
\sum_{j=1}^m x_j^\alpha, \qquad (\alpha>0, \ x_i\ge 0).
\end{equation}
Here the constants implied by the $O$ are absolute. Since the
splitting of $\{1, 2, \ldots, p\}$ into pairs can be done in
$\frac{p!}{(p/2)!} 2^{-p/2}$ different ways, we proved that
\begin{equation}\label{egnp}
\E \left(\sum_{k\in I} g(n_kx)\right)^p=
\begin{cases}
\frac{p!}{(p/2)!}2^{-\frac{p}{2}} \sigma_I^p
+O(T_N^*) \\
O (T_N^*)
\end{cases}
\end{equation}
according as $p$ is even or odd; here
$$
T_N^*= \exp (Cp^7) N^{(p-1)/2}(\log N)^p.
$$
Now, letting $G_I=\sum_{k\in I} g(n_kx)$, $R_I=\sum_{k\in I}
r(n_kx)$, we get, using the mean value theorem, H\"older's
inequality and (\ref{conv}),
\begin{align}\label{sngn}
&|\E S_I^p-\E G_I^p| \\
&\le \E|(G_I+R_I)^p-G_I^p|=\E |pR_I (G_I+\theta R_I)^{p-1}| \nonumber\\
&\le p\|R_I\|_p \|(G_I+\theta R_I)^{p-1}\|_{p/(p-1)}=p\|R_I\|_p
\|G_I+\theta R_I\|_p^{p-1}\nonumber \\
&\le 3p (\|G_I\|_p+3)^{p-1}\le 3p2^{p-2}(\|G_I\|_p^{p-1}+3^{p-1}),
\nonumber
\end{align}
for some $0 \le \theta=\theta (x) \le 1$. For even $p$ we get from
(\ref{egnp}), together with (\ref{conv}) with $\alpha=1/p$, that
\begin{equation*}\label{gn}
\|G_I\|_p\ll p\sigma_I + \exp (Cp^6) \sqrt{N} \log N.
\end{equation*}
For $p$ odd, we get the same bound, since $\|G_I\|_p \le
\|G_I\|_{p+1}$. Thus for any $p\ge 3$ we get from (\ref{sngn})
\begin{align*}
\E S_I^p &\le \E G_I^p+3p2^{p-2}(\|G_I\|_p^{p-1}+3^{p-1})\\
&\le \E G_I^p+3p2^{p-2} \left[  (p\sigma_I \exp
(Cp^6))^{p-1}+3^{p-1}\right]
O(1)^{p-1}\\
&\le EG_I^p + \exp (Cp^7) \left( C\sqrt{N}\right)^{p-1},
\end{align*}
completing the proof of Lemma \ref{lemma2}.

\begin{lem}\label{lem:2}  Let
$$f=\sum_{k=1}^d (a_k \cos 2\pi kx+b_k \sin 2 \pi kx)$$
be a trigonometric polynomial and let  $I$, $J$ be disjoint sets
of of positive integers with cardinality $M$ and $N$,
respectively, where $M/N\le C$ with a sufficiently small constant
$0<C<1$. Assume $\sigma_I\gg |I|^{1/2}$, $\sigma_J\gg |J|^{1/2}$.
Then for any integers $p\ge 2$, $q\ge 2$ we have
\begin{align}
\label{eq:1} &\E (S_I/\sigma_I)^p
(S_J/\sigma_J)^q =\\
&= \begin{cases} \dfrac{p!}{(p/2)!  2^{p/2}} \dfrac{q!}{(q/2)!
2^{q/2}} + O (T_{M,N})
&\text{if }\ p, q \text{ are even}\\
O(T_{M,N}) &\text{otherwise}
\end{cases}
\nonumber
\end{align}
where
$$
T_{M,N} = C^{p + q}_{p + q} \bigl(M^{-1/2} + (M/N)^{1/2}\bigr),
$$
and $C_p=\exp (cp^6)$ is the constant in Lemma \ref{evertse}.
\end{lem}

\begin{proof}
To simplify the formulas, we assume again that $f$ is a cosine
polynomial, i.e.\
$$
f(x) = \sum^{d}_{j = 1} a_j \cos 2\pi j x.
$$
The general case requires only trivial changes. Clearly
\begin{align*}
&S_I^p S_J^q =\frac1{2^{p + q}} \sum_{1\leq j_1, \dots, j_{p + q}
\leq d}
a_{j_1} \dots a_{j_{p + q}} \times\\
&\quad \times \sum_{\substack{k_1, \dots, k_p \in I\\
k_{p + 1}, \dots, k_{p + q} \in J}}\kern-12pt \cos 2\pi \bigl(\pm
j_1 n_{k_1} \pm \dots \pm j_{p + q} n_{k_{p + q}}\bigr)x
\end{align*}
and thus
\begin{align}
&\E S_I^p S_J^q=
\label{eq:2}\\
&= \frac1{2^{p + q }} \sum_{1 \leq j_1, \dots, j_{p + q} \leq d}
a_{j_1}\ldots a_{j_{p+q}}
\sum_{\substack{k_1, \dots, k_p \in I\\
k_{p + 1}, \dots, k_{p + q} \in J}} I\bigl\{\pm j_1 n_{k_1} \pm
\dots \pm j_{p + q} n_{k_{p + q}} = 0 \bigr\}. \nonumber
\end{align}
Assume that $j_1, \dots, j_{p + q}$ and the signs $\pm$ are fixed
and consider a solution of
\begin{equation}
\label{eq:3} \pm j_1 n_{k_1} \pm \dots \pm j_{p + q} n_{k_{p + q}}
= 0.
\end{equation}
Clearly, the set $\{1, 2, \dots, p + q\}$ can be split into
disjoint sets $A_1, \dots, A_\ell$ such that for each  such set
$A$ we have $\sum\limits_{i \in A} \pm j_i n_{k_i} = 0$ and no
further subsums of these sums are equal to~$0$. Call a set $A$
type $1$ or type $2$ according as $A$ intersects $\{1, 2, \ldots,
p \}$ or $A \subseteq \{p+1,\dots, p+q\}$.  Similarly as in the
proof of Lemma \ref{lemma2}, the number of solutions of the
equation $\sum_{i\in A} \pm j_i n_{k_i}=0$ is at most $C_{p + q-1}
M$ or $C_{p + q-1} N$ according as $A$ is of type $1$ or type~$2$.
Thus the number of solutions of \eqref{eq:3} belonging to a fixed
decomposition $\{A_1, \dots, A_\ell\}$ is at most
\begin{equation}
\label{eq:4} (C_{p + q-1} M)^R (C_{p + q-1} N)^S
\end{equation}
where $R$ and $S$ denote, respectively, the number of $A_i$'s with
type~$1$ and type~$2$. Let $R^*$ and $S^*$ denote the total
cardinality of sets of type~$1$ and type~$2$. Then $R = R^*/2$ or
$R \leq (R^* - 1)/2$ according as all sets of type~$1$ have
cardinality~$2$ or at least one of them has cardinality $\geq 3$.
A~similar statement holds for sets of type~$2$ and thus if there
exists at least one set $A_i$ with $|A_i| \geq 3$, the expression
in \eqref{eq:4} can be estimated as follows, using also $R^* + S^*
= p + q$, $S^* \leq q$,
\begin{align*}
(C_{p + q-1} M)^R (C_{p + q-1} N)^S & \leq (C_{p + q-1} M)^{R^*/2}
(C_{p + q-1} N)^{S^*/2}(C_{p+q-1} M)^{-1/2} \\
&= (C_{p + q-1} M)^{(p + q - S^*)/2} (C_{p + q-1} N)^{S^*/2}
(C_{p + q-1} M)^{-1/2} \\
&= (C_{p + q-1} M)^{(p + q)/2} (N/M)^{S^*/2} (C_{p + q-1} M)^{-1/2} \\
&\leq (C_{p + q-1} M)^{(p + q)/2} (N/M)^{q/2} (C_{p + q-1} M)^{-1/2} \\
&\leq C^{(p + q)/2}_{p + q} M^{p/2} N^{q/2} M^{-1/2} \\
&\ll C^{p+q} C^{(p + q)/2}_{p + q} \sigma_I^p \sigma_J^q M^{-1/2},
\end{align*}
where in the last step we used (\ref{anm}). Since the total number
of decompositions of the set $\{1, 2, \dots, p + q\}$ into subsets
is $\le (p + q)! 2^{p + q} \ll 2^{(p + q)^2}$, it follows that the
contribution of those solutions of \eqref{eq:3} in \eqref{eq:2}
where $|A_i| \geq 3$ for at least one set $A_i$ is
$$
\ll 2^{(p + q)^2} (\log d)^{p+q} C^{(p + q)/2}_{p + q}
M^{-1/2}\sigma_I^p \sigma_J^q.
$$

We now turn to the contribution of those solutions of \eqref{eq:3}
where all sets $A_1, \dots, A_\ell$ have cardinality~$2$. This can
happen only if $p + q$ is even and then $\ell = (p + q)/2$. Fixing
$A_1, \ldots A_\ell$, the sum of the corresponding terms in
\eqref{eq:2} can be written as
$$
2^{-(p+q)}\sum_{1 \leq j_1, \dots, j_{p + q}  \leq d}\kern-7pt
a_{j_1} \dots a_{j_{p + q}} I \biggl\{ \sum_{i \in A_1} \pm j_i
n_{k_i} = 0 \biggr\} \dots I \biggl\{ \sum_{i \in A_{(p +
q)/2}}\kern-4pt \pm j_i n_{k_i} = 0 \biggr\}
$$
and this is the product of $(p + q)/2$ such sums belonging to
$A_1, \dots, A_{(p + q)/2}$. For an $A_i \subseteq \{1, \dots,
p\}$ we get
$$
\frac{1}{4}\sum_{\substack{1 \leq i,j \leq d\\
k_i, k_j \in I}} a_i a_j I \bigl\{ \pm in_{k_i} \pm jn_{k_j} =
0\bigr\} = ES_I^2 = \sigma^2_I.
$$
Similarly, for any $A_i \subseteq \{p + 1, \dots, p + q\}$ the
corresponding sum equals $ES_J^2 = \sigma^2_J$. Finally, if a set
$A_i$ is ``mixed'', i.e.\ if one of its elements is in $\{1,
\dots, p\}$, the other in $\{p + 1, \dots, p + q\}$, then we get
$ES_IS_J := \sigma_{I,J}$ (cf.\ (\ref{eq:2}) with $p=q=1$). Thus,
if we have $t_1$ sets $A_i \subseteq \{1, \dots, p\}$, $t_2$ sets
$A_i \subseteq \{ p + 1, \dots, p + q\}$ and $t_3$ ``mixed'' sets,
we get $\sigma^{2t_1}_I \sigma^{2t_2}_J \sigma^{t_3}_{I,J}$.
Clearly $t_3 = 0$ can occur only if $p$ and $q$ are both even and
then $t_1 = p/2$, $t_2 = q/2$, i.e.\ we get $\sigma^p_I
\sigma^q_J$ which, taking into account the fact that $\{1, 2,
\ldots, p\}$ can be split into 2-element subsets in
$\frac{p!}{(p/2)!}2^{-p/2}$ different ways, gives the contribution
$$
\frac{p!}{(p/2)! 2^{p/2}} \frac{q!}{(q/2)! 2^{q/2}} \sigma^p_I
\sigma^q_J.
$$
Assume now that $t_3 = s$, $1 \leq s \leq p \land q$. Then $t_1 =
(p - s)/2$, $t_2 = (q - s)/2$; clearly if $p$ and  $q$ are both
even, then $s$ can be $0, 2,4,\dots$ and if $p$ and $q$ are both
odd, then $s$ can be $1,3,5,\dots$. Thus the contribution in this
case is
\begin{equation}
\label{eq:5} \sigma^{p - s}_I \sigma^{q - s}_J \sigma^s_{I,J}.
\end{equation}
From
$$
\sigma_{I,J} = \frac{1}{4} \sum_{\substack{1 \leq i, j \leq d\\
k \in I, \ell \in J}} a_i a_j I \bigl\{ \pm in_k \pm jn_\ell =
0\bigr\}
$$
we see that $\sigma_{I, J} \ll (|I| \land |J|) = M$ and thus
dividing with $\sigma^p_I \sigma^q_J$ and summing for $s$,
\eqref{eq:5} yields, using again (\ref{anm}),
\begin{align*}
\sum_{s \geq 1} \sigma^{-s}_I \sigma^{-s}_J \sigma^s_{I,J} &\le
\sum_{s \geq 1} C^{2s}(MN)^{-s/2} M^s  =\\
&= \sum_{s \geq 1} C^{2s}(M/N)^{s/2} \ll (M/N)^{1/2},
\end{align*}
provided $C$ is small enough.
\end{proof}

\begin{lem}
\label{lem:3} Under the conditions of Lemma \ref{lem:2} we have
for any $0 < \delta < 1$
\begin{align*}
&\left| \E\left(\exp\left( it S_I/\sigma_I + is S_J/\sigma_J
\right)\right) -
e^{-(t^2 + s^2)/2} \right| \ll \\
&\ll e^{-C(\log M)^\delta} + e^{C(\log M)^{7\delta}} \bigl(
M^{-1/2} + \sqrt{M/N}\bigr)
\end{align*}
for $|t|, |s| \leq \frac{1}{4} (\log M)^{\delta/2}$.
\end{lem}

Lemma \ref{lem:3} (and also Lemma \ref{lem:5} below) show that the
random variables $S_I/\sigma_I$ and $S_J/\sigma_J$ are
asymptotically independent if $|I|\to\infty$, $|J|\to\infty$,
$|I|/|J|\to 0$. Note that $I$ and $J$ are arbitrary disjoint
subsets of $\mathbb N$: they do not have to be intervals, or being
separated by some number $x\in {\mathbb R}$, they can be also
"interlaced". Thus $\{n_kx\}$ obeys an "interlaced" mixing
condition, an unusually strong near independence property
introduced by Bradley \cite{br}. Note that this property is
permutation-invariant, explaining the permutation-invariance of
the CLT and LIL in Theorem 1.

It is easy to extend Lemma \ref{lem:3} for the joint
characteristic function of normed sums $S_{I_1}/\sigma_{I_1},
\ldots S_{I_d}/\sigma_{I_d}$ of $d$ disjoint blocks $I_1, \ldots
I_d$, $d\ge 3$. Since, however, the standard mixing conditions
like $\alpha$-mixing, $\beta$-mixing, etc.\ involve pairs of
events and the present formulation will suffice for the CLT and
LIL for $f(n_{\sigma (k)}x)$, we will consider only the case
$d=2$.

\begin{proof}
Using $\left| e^{ix} - \sum^{k - 1}_{p = 0} \dfrac{(i x)^p}{p!}
\right|  \leq \dfrac{|x|^k}{k!}$, valid for any $x \in \mathbb R$,
$k \geq 1$ we get for any $L \geq 1$
\begin{align*}
\exp \left(it S_I/\sigma_I \right) &=
\sum^{L - 1}_{p = 0} \frac{(it)^p}{p!} ( S_I/\sigma_I)^p + \\
&\quad + \theta_L(t,x,I) \frac{|t|^L}{L!} |S_I/\sigma_I|^L =\\
&= : U_L(t,x,I) + \theta_L(t,x,I) \frac{|t|^L}{L!} |
S_I/\sigma_I|^L
\end{align*}
where $|\theta_L(t,x,I)| \leq 1$. Writing a similar expansion for
$\exp (is S_J/\sigma_J)$ and multiplying, we get
\begin{align*}
&\E\left(\exp ( it S_I/\sigma_I
+ is S_J/\sigma_J)\right) = \\
&= \E \bigl(U_L(t,x,I) U_L(s,x,J) \bigr) + \E \left( U_L(t,x,I)
\theta_L(s,x,J) \frac{|s|^L}{L!} |S_J/\sigma_J|^L \right) + \\
&\quad + \E \left(U_L(s,x,J) \theta_L (t,x,I) \frac{|t|^L}{L!}
|S_I/\sigma_I|^L\right) +\\
&\quad + \E \left( \frac{|t|^L}{L!} \frac{|s|^L}{L!}
|S_I/\sigma_I|^L |S_J/\sigma_J|^L \theta_L (t, x, I) \theta_L(s,x,J) \right) =\\
&= I_1 + I_2 + I_3 + I_4.
\end{align*}
We estimate $I_1$, $I_2$, $I_3$, $I_4$ separately. We choose $L =
2 \bigl[(\log M)^\delta\bigr]$ and use  Lemma~\ref{lem:2} to get
\begin{align*}
I_1 &= \sum^{L - 1}_{\substack{p,q = 0\\ p,q \text{ even}}}
\frac{(it)^p}{(p/2)! 2^{p/2}} \frac{(is)^q}{(q/2)! 2^{q/2}} +\\
&\quad + O(1) \sum^{L - 1}_{p,q = 0} \frac{|t|^p}{p!}
\frac{|s|^q}{q!} C^{2L}_{2L} \bigl(M^{-1/2} + (M/N)^{1/2} \bigr) \\
&=: I_{1,1} + I_{1,2} .
\end{align*}
Here
\begin{align*}
&e^{-(t^2 + s^2)/2} - I_{1,1} =\\
&=\! \biggl( \underset{p \text{ even}}{\sum\limits^{L - 1}_{p =
0}} \! \frac{(it)^p}{(p/2)! 2^{p/2}} \biggr)\! \biggl( \underset{q
\text{ even}}{\sum\limits^{\infty}_{q = L}} \!
\frac{(is)^q}{(q/2)! 2^{q/2}} \biggr)\! +\! \biggl( \underset{p
\text{ even}}{\sum\limits^{\infty}_{p = L}} \!
\frac{(it)^p}{(p/2)! 2^{p/2}} \biggr)\! \biggl( \underset{q \text{
even}}{\sum\limits^{\infty}_{q = 0}} \! \frac{(is)^q}{(q/2)!
2^{s/2}} \biggr).
\end{align*}
Using $n!\ge (n/3)^n$ and $t^2 \leq L/24 \leq p/24$ we get
\begin{align*}
\Biggl| \underset{p \text{ even}}{\sum^\infty_{p = L}}
\frac{(it)^p}{(p/2)! 2^{p/2}}\Biggr| &\leq \sum^\infty_{p = L}
\frac{|t|^p}{(p / 3)^{p/2}} =
\sum^\infty_{p = L} \biggl(\frac{t^2}{p/3}\biggr)^{p/2} \leq\\
&\leq \sum^\infty_{p = L} \left(\frac{1}{4}\right)^{p/2} \leq 2
\cdot 2^{-L} \leq 8e^{-(\log M)^\delta}
\end{align*}
and similarly
$$
\Biggl| \underset{q \text{ even}}{\sum\limits^\infty_{q = L}}
\frac{(is)^q}{(q/2)! 2^{q/2}} \Biggr| \ll e^{-(\log M)^\delta}.
$$
Thus
$$
\Biggl| \underset{p \text{ even}}{\sum\limits^{L - 1}_{p = 0}}
\frac{(it)^p}{(p/2)! 2^{p/2}} \Biggr| \leq e^{-t^2/2} + \Biggl|
\underset{p \text{ even}} {\sum\limits^\infty_{p = L}}
\frac{(it)^p}{(p/2)! 2^{p/2}} \Biggr| \leq 9,
$$
and a similar estimate holds for
$$
\sum^{L - 1}_{\substack{q = 0\\ q \text{ even}}}
\frac{(is)^q}{(q/2)! 2^{q/2}}.
$$
Consequently
$$
\left| I_{1,1} - e^{-(t^2 + s^2)/2}\right| \ll  e^{-(\log
M)^\delta}.
$$
On the other hand,
\begin{align*}
|I_{1,2}| &\ll \left(\sum_{p=0}^\infty \frac{|t|^p}{p!}\right)
\left(\sum_{q=0}^\infty \frac{|s|^q}{q!}\right) C^{2L}_{2L}
\bigl(M^{-1/2}
+ (M/N)^{1/2}\bigr) \\
&\ll e^{|t|+|s|} e^{C(2L)^7} \bigl(M^{-1/2}
+ (M/N)^{1/2}\bigr)  \\
&\ll e^{C(\log M)^{7\delta}} \bigl(M^{-1/2} + (M/N)^{1/2}\bigr).
\end{align*}
Thus we proved
$$
\bigl|I_1 - e^{-(t^2 + s^2)/2}\bigr| \ll e^{-C(\log M)^\delta} +
e^{C(\log M)^{7\delta}} \bigl(M^{-1/2} + (M/N)^{1/2}\bigr).
$$
Next we estimate $I_4$. Using Lemma~\ref{lem:2} and $t^2\le L/24$
we get, since $L$ is even,
\begin{align*}
I_4 &\le \frac{|t|^L}{L!} \frac{|s|^L}{L!} \E |S_I/\sigma_I|^L |S_J/\sigma_J|^L  \\
&\ll \frac{|t|^L}{L!} \frac{|s|^L}{L!} \left[
\biggl(\frac{L!}{(L/2)! 2^{L/2}} \biggr)^2 + C^{2L}_{2L}
\bigl(M^{-1/2}
+ (M/N)^{1/2} \bigr)\right]  \\
&\ll \frac{|t|^L |s|^L}{((L/2)!)^2} + \frac{|t|^L |s|^L}{(L!)^2}
C^{2L}_{2L} \bigl(M^{-1/2} + (M/N)^{1/2} \bigr) \\
&\ll \left(\frac{t^2}{L/6}\right)^{L/2} \left(\frac{s^2}{L/6}\right)^{L/2} \\
&\quad + \left(\frac{t^2}{L/6}\right)^{L/2}
\left(\frac{s^2}{L/6}\right)^{L/2} e^{C(2L)^7}
\bigl(M^{-1/2} + (M/N)^{1/2} \bigr)  \\
&\ll 4^{-L} + 4^{-L} e^{C(2L)^7} \bigl(M^{-1/2}
+ (M/N)^{1/2} \bigr) \\
&\ll e^{-C(\log M)^\delta} + e^{C(\log M)^{7\delta}}
\bigl(M^{-1/2} + (M/N)^{1/2} \bigr).
\end{align*}
Finally we estimate $I_2$ and $I_3$. Clearly
\begin{align*}
|U_L(t,x,I)| &\leq | \exp ( it S_I/\sigma_I ) | + \frac{|t|^L}{L!} |S_I/\sigma_I|^L \\
&\leq 1 + \frac{|t|^L}{L!} |S_I/\sigma_I|^L
\end{align*}
and thus
\begin{align*}
|I_2| &\leq \E \left(\frac{|s|^L}{L!} |S_J/\sigma_J|^L\right)
 + \E \left(\frac{|t|^L}{L!} \frac{|s|^L}{L!}
 |S_I/\sigma_I|^L |S_J/\sigma_J|^L\right).
\end{align*}
Here the second summand can be estimated exactly in the same way
as $I_4$ and the first one can be estimated by using Lemma
\ref{lemma2}. Thus we get
$$
|I_2| \ll e^{-C(\log M)^\delta} + e^{C(\log M)^{7\delta}}
\bigl(M^{-1/2} + (M/N)^{1/2} \bigr).
$$
A similar bound holds for $I_3$ and this completes the proof of
Lemma~\ref{lem:3}.
\end{proof}

\begin{lem}
\label{lem:4} Let $F$ and $G$ be probability distributions on
$\mathbb R^2$  with characteristic functions $\varphi$ and
$\gamma$, respectively and let $T>0$. Then there exists a
probability distribution $H$ on $\mathbb R^2$ such that
$H\left(|x| \geq T^{-1/2} \log T\right) \ll e^{-T/2}$ and for any
Borel set $B \subset [-T, T]^2$
$$
\bigl|(F * H)(B) - (G * H)(B)\bigr| \ll T^2 \int\limits_{[-T,T]^2}
|\varphi(u) - \gamma(u)|du + e^{-(\log T)^2/4}.
$$
The constants implied by $\ll$ are absolute.
\end{lem}

\begin{proof}
Let $\zeta_0$ be a standard $N(0, {\bf I})$ random variable in
${\mathbb R}^2$ and $\zeta= \frac{\log T}{T} \zeta_0$. Clearly we
have
$$
P\left(|\zeta| \geq \frac{\log T}{\sqrt T}\right) =
P\bigl(|\zeta_0| \geq \sqrt{T} \bigr) \ll e^{-T/2}.
$$
Letting $\psi$ and $H$ denote, respectively, the characteristic
function and distribution of $\zeta$, we get
\begin{align*}
\bigl|f_{F * H} (x) - f_{G * H}(x)\bigr| &\leq (2\pi)^{-1}
\int\limits_{\mathbb R^2} |\varphi(u) - \gamma(u) |\,
|\psi (u)| du \leq \\
&\leq \int\limits_{[-T, T]^2} |\varphi (u) - \gamma (u)| du + 2
\int\limits_{u \notin [-T, T]^2} |\psi (u)|du,
\end{align*}
where $f_{F * H}$, $f_{G * H}$ denote  the density functions
corresponding to the distributions $F * H$ and $G * H$,
respectively. Letting $\tau = T^{-1} \log T$, we clearly have
$\psi (u)= e^{-\tau^2|u|^2/2}$ for $u\in {\mathbb R}^2$ and a
simple calculation shows
$$\int_{u \notin [-T, T]} |\psi(u)|du \ll e^{-(\log T)^2/3}.$$
Thus
$$
\bigl| f_{F * H} (x) - f_{G * H} (x) \bigr| \ll \int\limits_{[-T,
T]^2} |\varphi (u) - \gamma (u)| du + e^{-(\log T)^2 / 3}
\quad\text{for all} \ x\in {\mathbb R}^2
$$
whence for $B \subseteq [-T, T]^2$ we get
$$
|(F * H)(B) - (G * H)(B) | \ll T^2 \int\limits_{[-T, T]}
 |\varphi (u) - \gamma (u)|du + T^2 e^{-(\log T)^2/3},
$$
proving Lemma~\ref{lem:4}.
\end{proof}

\begin{lem}
\label{lem:5} Under the conditions of Lemma \ref{lem:2}  we have
for any $0<\delta<1$ and for $|x|, |y| \leq \frac{1}{8} (\log
M)^{\delta/2}$,
\begin{align}
\label{eq:6} &\left| P \left(S_I/\sigma_I \leq x, \,
S_J/\sigma_J \leq y\right) - \Phi (\hat x) \Phi (\hat y)\right| \ll \\
&\ll e^{-C(\log\log M)^2} + e^{C(\log M)^{7\delta}} \bigl(M^{-1/2}
+ (M/N)^{1/2} \bigr) \nonumber
\end{align}
where $\Phi$ is the standard normal distribution function and
$\hat x$, $\hat y$ are suitable numbers with $|\hat x - x| \leq
C(\log M)^{-\delta/8}$, $|\hat y - y| \leq C(\log M)^{-\delta /
8}$.
\end{lem}

\begin{proof}
Let
$$
F = \text{\rm dist} \left( S_I/\sigma_I, S_J/\sigma_J\right),
\quad G = N(0, I), \quad T = (\log M)^{\delta/2}.
$$
By Lemmas \ref{lem:3} and \ref{lem:4} we have for any Borel set $B
\subseteq [-T, T]^2$
\begin{align*}
&\bigl| (F * H)(B) - (G * H)(B) \bigr| \\
&\ll T^2 \int\limits_{[-T,T]^2} |\varphi (u) - \gamma (u)| du +
e^{-(\log
T)^2/4} \\
&\ll (\log M)^{2\delta} \Bigl[ e^{-C(\log M)^\delta} + e^{C(\log
M)^{7\delta}} \bigl(M^{-1/2} + (M/N)^{1/2}\bigr)\Bigr]
+ e^{-c(\log\log M)^2} \\
&\ll e^{-C(\log M)^\delta} + e^{C(\log M)^{7\delta}}
\bigl(M^{-1/2} + (M/N)^{1/2} \bigr) + e^{-C(\log\log M)^2}
\end{align*}
where $H$ is a distribution on $\mathbb R^2$ such that
$$
H\bigl(x : |x| \geq C(\log M)^{-\delta/8} \bigr) \leq e^{-C(\log
M)^{\delta/2}}.
$$
Applying Lemma \ref{lemma2} with $p=2[\log\log M]$ and using the
Markov inequality, we get
\begin{align*}
&P ( |S_I/\sigma_I|\ge T)\le T^{-p}\E (|S_I/\sigma_I|^p)\ll
T^{-p}\frac{p!}{(p/2)!} 2^{-p/2}\\
& \ll 4^p(\log M)^{-\delta p/2} p^p=4^p \exp
\left(p\log p-\frac{\delta p}{2}\log\log M\right)\\
&\ll \exp (-C(\log\log M)^2)
\end{align*}
and a similar inequality holds for $P(|S_J/\sigma_J|\ge T)$.
Convolution with $H$ means adding an (independent) r.v.\ which is
$< C(\log M)^{-\delta/8}$ with the exception of a set with
probability $e^{-C(\log M)^{\delta/2}}$. Thus choosing $B = [-T,x]
\times [-T, y]$ with $|x| \leq C(\log M)^{\delta/2}$, $|y| \leq
C(\log M)^{\delta/2}$ we get
\begin{align}
\label{eq:7} &\left| P \left( S_I/\sigma_I \le x, \,
S_J/\sigma_J  \le y \right) - \Phi(\hat x) \Phi(\hat y)\right|\ll \\
&\ll e^{-C(\log\log M)^2} + e^{C(\log M)^{7\delta}} \bigl(M^{-1/2}
+ (M/N)^{1/2}\bigr) \nonumber
\end{align}
where $|\hat x - x| \leq C(\log M)^{-\delta/8}$, $|\hat y - y|
\leq C(\log M)^{-\delta/8}$.
\end{proof}

\bigskip\noindent
{\bf Remark.}  The one-dimensional analogue of Lemma \ref{lem:5}
can be proved in the same way (in fact, the argument is much
simpler):
$$|P(S_I/\sigma_I \le x)-\Phi (\hat{x})|\ll e^{-C(\log\log M)^2}$$
for $|x|\le \frac{1}{8} (\log M)^{\delta/2}$, where $|\hat x - x|
\leq C(\log M)^{-\delta/8}$. Using this fact, the statement of
Lemma \ref{lem:5} and simple algebra show that for $|x|, |y| \leq
\frac{1}{8} (\log M)^{\delta/2}$ we have
\begin{align*}
\label{eq:6} &\left| P \left(S_I/\sigma_I > x, \,
S_J/\sigma_J > y\right) - \Psi (\hat x) \Psi (\hat y)\right| \ll \\
&\ll e^{-C(\log\log M)^2} + e^{C(\log M)^{7\delta}} \bigl(M^{-1/2}
+ (M/N)^{1/2} \bigr) \nonumber
\end{align*}
where $\hat x$, $\hat y$ are suitable numbers with $|\hat x - x|
\leq C(\log M)^{-\delta/8}$, $|\hat y - y| \leq C(\log M)^{-\delta
/ 8}$. Here $\Psi (x)=1-\Phi (x)$.

\section{Proof of Theorem 1}

\bigskip
The CLT (\ref{fclt}) in Theorem \ref{th1} follows immediately from
Lemma \ref{lem:5}; see also the remark after Lemma \ref{lem:5}. To
prove the LIL (\ref{lilperm}), assume the conditions of Theorem
\ref{th1} and let $\sigma: {\mathbb N}\to {\mathbb N}$ be a
permutation of $\mathbb N$. Clearly for $p=O(\log\log N)$ we have
$\exp (Cp^7)\ll N^{1/4}$ and thus Lemma \ref{lemma2} implies
$$
\int_0^1 \left(\sum_{k=M+1}^{M+N} f(n_{\sigma(k)}x\right)^{2p} dx
\sim \frac{(2p)!}{p!}2^{-p} (1+O(N^{-1/4})) A_{N, M}^p \qquad
\text{as} \ N\to\infty
$$
uniformly for $p = O(\log\log N)$ and $M \ge 1$. Using this fact,
the upper half of the LIL (\ref{lilperm})  can be proved by
following the classical proof of Erd\H{o}s and G\'al \cite{eg} of
the LIL for lacunary trigonometric series. (The observation that
the upper half of the LIL follows from asymptotic moment estimates
was already used by Philipp \cite{ph1969} to prove the LIL for
mixing sequences.) To prove the lower half of the LIL we first
observe that the upper half of the LIL and relation
(\ref{starstar}) imply
\begin{equation}\label{fnorm}
 \limsup_{N\to\infty}\,(N\log\log N)^{-1/2} \sum_{k=1}^N f(n_{\sigma(k)}x)
\le K \|f\|^{1/8}  \qquad \text{a.e.}
\end{equation}
where $K$ is a constant depending on the generating elements of
$(n_k)$. Given any $f$ satisfying (\ref{fcond}) and
$\varepsilon>0$,  $f$ can be written as $f=f_1+f_2$ where $f_1$ is
a trigonometric polynomial and $\|f_2\|\le \varepsilon$, and thus
applying (\ref{fnorm}) with $f=f_2$ it is immediately seen that it
suffices to prove the lower half of the LIL for trigonometric
polynomials $f$.

Let $\theta \ge 2$ be an integer and set
$$
\eta_n = \frac{X_{\theta^n + 1} + \dots + X_{\theta^{n + 1}}}
{\gamma_n}
$$
where $X_j = f(n_{\sigma(j)} x)$, $\gamma^2_n = \text{\rm Var}
\bigl(X_{\theta^n + 1} + \dots + X_{\theta^{n + 1}}\bigr)$. Fix
$\varepsilon > 0$ and put
$$
A_n = \bigl\{ \eta_n \geq (1 - \varepsilon) (2\log\log
\gamma_n)^{1/2} \bigr\}.
$$
We will prove that $P(A_n \ \text{i.o.})=1$;  we use here an idea
of R\'ev\'esz \cite{re} and the following generalization of the
Borel-Cantelli lemma, see Spitzer \cite{sp}, p.\ 317.

\begin{lem}\label{renyi}
Let $A_n$, $n=1, 2, \ldots$ be events satisfying
$\sum_{n=1}^\infty P(A_n)=\infty$ and
$$
\lim_{N \to \infty} \frac{\sum_{1 \leq m, n \leq N} \bigl|P(A_m
\cap A_n) - P(A_m) P(A_n)\bigr|}{\Bigl(\sum^N_{n = 1} P(A_n)
\Bigr)^2} = 0.
$$
Then $P(A_n \ \text{i.o.})=1.$
\end{lem}

By the one-dimensional version of Lemma \ref{lem:5} (see the
remark at the end of Section 2) we have
\begin{equation}\label{star0}
P(A_n) = \Psi\bigl((1 - \varepsilon) (2 \log\log
\gamma_n)^{\frac12} + z_n \bigr) + O(e^{-C(\log n)^2})
\end{equation}
where $|z_n| \leq Cn^{-\delta/8}$. By the mean value theorem,
$\Psi (x)\sim (2\pi)^{-1/2}x^{-1} \exp(-x^2/2)$ and
$\theta^n \ll \gamma^2_n \ll \theta^n$ we have
\begin{align}\label{star1}
&\Psi\bigl((1 - \varepsilon) (2\log\log \gamma_n)^{1/2} + z_n\bigr)\\
&= \Psi\bigl((1 - \varepsilon)(2 \log\log \gamma_n)^{1/2}\bigr) +
\exp \left(-\frac{1}{2} \left[ (1-\varepsilon)(2\log\log
\gamma_n)^{1/2} +O(1)n^{-\delta/8}\right]^2
\right) O(n^{-\delta/8})\nonumber\\
&= \Psi\bigl((1 - \varepsilon)(2 \log\log \gamma_n)^{1/2}\bigr)+
\exp\left( -(1-\varepsilon)^2 \log\log \gamma_n +O(1)\right)O(n^{-\delta/8})\nonumber\\
&= \Psi\bigl((1 - \varepsilon)(2 \log\log \gamma_n)^{1/2}\bigr)+
O(1)\Psi \left( (1-\varepsilon)(2\log\log \gamma_n)^{1/2}\right)
(\log n)^{1/2}n^{-\delta/8}. \nonumber
\end{align}
In particular,
$$
\Psi\bigl((1 - \varepsilon) (2\log\log \gamma_n)^{1/2} +
z_n\bigr)\sim \Psi\bigl((1 - \varepsilon) (2\log\log
\gamma_n)^{1/2} \bigr)
$$
and thus (\ref{star0}) implies
\begin{equation}\label{star2} P(A_n) \sim  \Psi\bigl((1 - \varepsilon) (2\log\log
\gamma_n)^{1/2} \bigr) \gg  \frac{1}{n^{(1 - \varepsilon)^2}(\log
n)^{1/2}}.
\end{equation}
Hence the estimates in (\ref{star1}) yield
\begin{equation}\label{star3}
\Psi\bigl((1 - \varepsilon) (2\log\log \gamma_n)^{1/2} +
z_n\bigr)= \Psi\bigl((1 - \varepsilon) (2\log\log
\gamma_n)^{1/2}\bigr)+ O(P(A_n) n^{-\delta/16}).
\end{equation}
Now by Lemma \ref{lem:5} for $m\le n$ (see the Remark at the end
of Section 2)
\begin{align}\label{star4}
P(A_m \cap A_n) &= \Psi\bigl((1 - \varepsilon) (2\log\log
\gamma_m)^{1/2}
+ z_1 \bigr) \Psi\bigl((1 - \varepsilon) (2 \log\log \gamma_n)^{1/2} + z_2\bigr) \\
&\quad + O(1) \Bigl[ e^{-C(\log m)^2} + e^{Cm^{7\delta}} (e^{-Cm}
+ e^{-C(n - m)})\Bigr],\nonumber
\end{align}
provided $\log n \le m^{\delta/2}$. The expression $\Psi(\dots)
\Psi(\dots)$ in (\ref{star4}) equals by \eqref{star2},
\eqref{star3},
\begin{align*}
&\Psi\bigl((1 - \varepsilon) (2\log\log \gamma_m)^{1/2} \bigr)
\Psi\bigl((1 - \varepsilon) (2 \log\log \gamma_n)^{1/2} \bigr) +
O\left(P(A_m)P(A_n)m^{-\delta/16}\right).
\end{align*}
Hence, assuming also $n-m\ge m^{8\delta}$ we get from
(\ref{star4}),
\begin{align}\label{star5}
P(A_m \cap A_n) &= \Psi\bigl((1 - \varepsilon)(2 \log\log
\gamma_m)^{1/2}
\bigr) \Psi\bigl((1 - \varepsilon) (2 \log\log \gamma_n)^{1/2} \bigr) \\
&\quad + O\left( P(A_m)P(A_n) m^{-\delta/16}\right)  +
O\left(e^{-C(\log m)^2}\right). \nonumber
\end{align}
Further, by  \eqref{star0} and the above estimates
\begin{align*}
P(A_m) P(A_n) &= \Psi\bigl((1 - \varepsilon)(2\log\log
\gamma_m)^{1/2} \bigr)
\Psi\bigl((1 - \varepsilon)(2\log\log \gamma_n)^{1/2}\bigr) \\
&\quad + O\left(P(A_m) P(A_n) m^{-\delta/16}\right) + O\left(
e^{-C(\log m)^2}\right),
\end{align*}
and thus we obtained
\begin{lem}
\label{lem:6} We have
\begin{align*}
&\bigl|P(A_m \cap A_n) - P(A_m) P(A_n)\bigr| \ll P(A_m)P(A_n)
m^{-\delta/16} + e^{-C(\log m)^2}
\end{align*}
provided $n - m \geq m^{8\delta}$ and  $\log n \le m^\delta$.
\end{lem}

We can now prove

\begin{lem}
\label{lem:7} We have
$$
\lim_{N \to \infty} \frac{\sum_{1 \leq m, n \leq N} \bigl|P(A_m
\cap A_n) - P(A_m) P(A_n)\bigr|}{\Bigl(\sum^N_{n = 1} P(A_n)
\Bigr)^2} = 0.
$$
\end{lem}

\begin{proof}
By Lemma~\ref{lem:6} we have
\begin{align*}
&\sum_{\substack{1 \leq m, n \leq N\\
n - m \geq m^{8\delta}\\
m \geq CN^\delta}}
\bigl| P(A_m \cap A_n) - P(A_m) P(A_n)\bigr| \\
&\ll \biggl(\sum^N_{m = 1} P(A_m) m^{-\delta/16}\biggr) \biggl(
\sum^N_{n = 1}
P(A_n) \biggr) + N^2 e^{-c(\log N)^2} \\
& =o_N(1) \biggl(\sum^N_{m = 1} P(A_m) \biggr) \biggl(\sum^N_{n =
1} P(A_n) \biggr) +O(1) =o_N(1)\biggl(\sum^N_{m = 1} P(A_m)
\biggr)^2,
\end{align*}
since $\sum\limits^N_{n = 1} P(A_n) = +\infty$ by \eqref{star2}.
Further
\begin{align*}
&\sum_{\substack{1 \leq m, n \le N\\
0 \leq n - m \leq m^{8\delta}}} \bigl| P(A_m \cap A_n) - P(A_m)
P(A_n)\bigr|\\
&\leq \sum_{\substack{1 \leq m, n \leq N\\
m \leq n \leq m + m^{8\delta}}} 2P(A_m) \leq 2 \sum^N_{m = 1}
m^{8\delta} P(A_m) \leq 2N^{8\delta} \sum^N_{m = 1} P(A_m)
\end{align*}
and
$$
\sum_{\substack{1 \leq m \leq n \leq N\\
m \leq CN^\delta}} \bigl| P(A_m \cap A_n) - P(A_m) P(A_n) \bigr|
\leq \sum_{\substack{1 \leq m \leq n \leq N\\
m \leq CN^\delta}} 2P(A_n) \leq 2N^\delta \sum^N_{n = 1} P(A_n).
$$
The previous estimates imply
\begin{align*}
&\sum_{1 \leq m, n \leq N} \bigl| P(A_m \cap A_n) - P(A_m) P(A_n)
\bigr| \\
& \qquad \ll o_N(1) \left(\sum^N_{n = 1} P(A_n)\right)^2 +
 N^{8\delta}\left(\sum^N_{n = 1} P(A_n)\right).
\end{align*}
Since
$$
\sum^N_{n = 1} P(A_n) \gg \sum^N_{n = 1} \frac{1}{ n^{(1 -
\varepsilon)^2}(\log n)^{1/2} \,} \gg \sum^N_{n = 1} \frac1{n^{1 -
\varepsilon}} \gg N^\varepsilon,
$$
choosing $\delta < \varepsilon/8$ we get
$$
\frac{\sum_{1 \leq m, n \leq N} \bigl| P(A_m \cap A_n) - P(A_m)
P(A_n)\bigr|} {\Bigl(\sum^N_{n = 1} P(A_n) \Bigr)^2} \ll o_N(1)+
\frac{N^{8\delta}} {\sum^N_{n = 1} P(A_n)} \to 0.
$$
\end{proof}

We can now complete the proof of the lower half of the LIL. By
Lemmas \ref{renyi} and \ref{lem:7} and  we have with probability 1
\begin{equation}\label{ls}
\bigl | X_{\theta^n + 1} + \dots + X_{\theta^{n + 1}} \bigr | \ge
(1-\varepsilon)(2\gamma_n^2 \log\log \gamma_n)^{1/2}
\qquad\text{i.o.}
\end{equation}
where $\gamma_n=\bigl\| X_{\theta^n + 1} + \dots + X_{\theta^{n +
1}} \bigr\|$. By the already proved upper half of the LIL we have
\begin{equation}\label{us}
\bigl | X_{1} + \dots + X_{\theta^{n}} \bigr | \le
(1+\varepsilon)(2A_{\theta^n}^2 \log\log A_{\theta^n})^{1/2}
\qquad\text{a.s.}
\end{equation}
and  (\ref{starstar}) and the assumptions of Theorem \ref{th1}
imply
\begin{equation}\label{theta} A_{\theta^{n+1}}/A_{\theta^n}\ge C
\theta^{1/2},
\end{equation}
whence
\begin{align}\label{thet}
\gamma_n & \geq \bigl\| X_1 + \dots + X_{\theta^{n + 1}} \bigr\|
-\bigl\| X_1 + \dots +
X_{\theta^n}\bigr\| =A_{\theta^{n+1}}-A_{\theta^n}\\
& \ge A_{\theta^{n+1}}\left( 1-O(\theta^{-1/2})\right). \nonumber
\end{align}
Thus using (\ref{ls}), (\ref{us}), (\ref{theta}) and (\ref{thet})
we get with probability 1 for infinitely many $n$
\begin{align*}
&\bigl | X_{1} + \dots + X_{\theta^{n+1}}
\bigr | \\
&\ge (1-\varepsilon)(2\gamma_n^2 \log\log
\gamma_n)^{1/2}-(1+\varepsilon)
(2A_{\theta^n}^2 \log\log A_{\theta^n})^{1/2}\\
& \ge (1-2\varepsilon)
(2A_{\theta^{n+1}}^2 \log\log A_{\theta^{n+1}})^{1/2},
\end{align*}
provided we choose $\theta=\theta(\varepsilon)$ large enough. This
completes the proof of the lower half of the LIL.

To prove the Corollary, assume that
\begin{equation}\label{limG}
N^{-1/2}\sum_{k=1}^N f(n_{\sigma(k)}x)
\overset{d}{\longrightarrow} G
\end{equation}
with a nondegenerate distribution $G$. By Lemma \ref{lemma2} and
(\ref{starstar}) we have
$$
\E \left(\sum_{k=1}^N f(n_{\sigma(k)}x)\right)^4 \ll N^2,
$$
and thus the sequence
$$N^{-1}\left(\sum_{k=1}^N f(n_{\sigma(k)}x)\right)^2, \qquad N=1, 2,
\ldots$$ is bounded in $L_2$ norm and consequently uniformly
integrable. Thus the second moment of the left hand side of
(\ref{limG}) converges to the second moment $\gamma^2$ of $G$,
which is nonzero, since $G$ is nondegenerate. Thus we proved
(\ref{gamma1}), and since the nonnegativity of the Fourier
coefficients of $f$ implies (\ref{anm}), Theorem \ref{th1} yields
(\ref{fclt2}) and (\ref{lilperm2}).

In conclusion, we prove the remark made at the end of the
Introduction concerning the set $\Gamma_f$ of limiting variances
corresponding to  all permutations $\sigma$. Let $f$ be a function
satisfying (\ref{fcond})
with nonnegative Fourier coefficients. Assume that $f$ is even,
i.e. its Fourier series
$$
f(x) \sim \sum_{j=1}^\infty a_j \cos 2 \pi j x
$$
is a pure cosine series; the general case requires only trivial
changes.
Note that the Fourier coefficients of $f$ satisfy
(\ref{fouriercoeffs})
and by Kac \cite{ka} we have
$$
\int_0^1 \left( \sum_{k=1}^N f\left(2^k x\right) \right)^2 dx
\sim \gamma_f^2 N
$$
where $$\gamma_f^2= \|f\|^2+2\sum_{r=1}^\infty \int_0^1
f(x)f(2^rx)dx \ge \|f\|^2.
$$
We first note that for any permutation $\sigma: {\mathbb N}\to
{\mathbb N}$ we have
\begin{equation}\label{lsa}
\|f\|^2\le \frac{1}{N} \int_0^1 \left( \sum_{k=1}^N
f(2^{\sigma(k)}x)\right)^2\, dx \le \gamma_f^2
\end{equation}
for any $N\ge 1$. To see this, we observe that
\begin{align}\label{ls2}
&\int_0^1 \left( \sum_{k=1}^N f(2^{\sigma(k)}x)\right)^2\, dx\\
&= N\|f\|^2+ \sum_{1\le i\ne j\le N} \int_0^1 f(2^{\sigma(i)}x)
f(2^{\sigma(j)}x)\,dx\nonumber\\
&=N\|f\|^2+ \sum_{1\le i \ne j\le N} \int_0^1 f(x)
f(2^{|\sigma(j)-\sigma(i)|}x)\,dx\nonumber\\
&=N\|f\|^2+ \sum_{r=1}^\infty a_r^{(N)} \int_0^1 f(x) f(2^r
x)\,dx\nonumber
\end{align}
where
$$a_r^{(N)}=\#\{1\le i\ne j\le N: |\sigma(j)-\sigma(i)|=r\}.$$
Fix $r\ge 1$. Clearly, for any $1\le i \le N$ there exist at most
two indices $1\le j\le N$, $j\ne i$ such that
$|\sigma(j)-\sigma(i)|=r$. Hence $a_r^{(N)}\le 2N$ and by the
nonnegativity of the Fourier coefficients of $f$, the integrals in
the last line of (\ref{ls2}) are nonnegative. Thus (\ref{lsa}) is
proved. Next we claim that for any $\rho \in \left[ \|f\|,
\gamma_f \right]$ we can find a permutation $\sigma: \mathbb{N}
\to \mathbb{N}$ such that
\begin{equation}\label{permint}
\int_0^1 \left( \sum_{k=1}^N f\left(2^{\sigma(k)} x\right)
\right)^2 dx \sim \rho^2 N. \end{equation} To this end, we will
need
\begin{lem} \label{lemmavar} For some $J \geq 0$ let
$$
g(x)=\sum_{j=J+1}^\infty a_j \cos 2 \pi j x.
$$
Then for any set $\{m_1, \dots, m_N\}$ of distinct positive
integers we have
$$
\left\|\sum_{k=1}^N g\left(2^{m_k}x\right) \right\| \leq \left\{
\begin{array}{ll} 2\sqrt{N} &\textrm{for}~ J=0 \\
\sqrt{2N}J^{-1/2} & \textrm{for}~J \geq 1. \end{array}\right.
$$
\end{lem}

\noindent  \emph{Proof:~} Similarly to (\ref{esnp}) we have
\begin{eqnarray*}
& & \int_0^1 \left( \sum_{k=1}^N g\left(2^{m_k}x\right) \right)^2 dx \\
& = & \frac{1}{2} \sum_{1 \leq k_1,k_2 \leq N} \sum_{j_1,j_2 \geq
J+1} a_{j_1}a_{j_2} \cdot \mathbf{1}\left(j_1 2^{m_{k_1}}
= j_2 2^{m_{k_2}} \right) \\
& \leq &  \sum_{1 \leq k_1 \leq k_2 \leq N} \sum_{j_1,j_2 \geq
J+1} \frac{1}{j_1 j_2} \cdot \mathbf{1}\left(j_1 2^{m_{k_1}} = j_2
2^{m_{k_2}} \right) \\
& \leq &  \sum_{1 \leq k_1 \leq k_2 \leq N} \sum_{j \geq J+1}
\frac{2^{m_{k_1}}}{j^2 2^{m_{k_2}}} \\
& \leq & N \sum_{v=0}^\infty 2^{-v} \sum_{j \geq J+1} \frac{1}{j^2} \\
& \leq & \left\{ \begin{array}{ll} 4N &\textrm{for}~ J=0 \\
2NJ^{-1} & \textrm{for}~J \geq 1. \end{array}\right. \quad
\square\\
\end{eqnarray*}

Let now $\rho\in [\|f\|, \gamma_f]$ be given and write
$$
\alpha = \frac{\rho^2-\|f\|^2}{\gamma_f^2 -\|f\|^2}.
$$
Clearly
$$
\alpha \in [0,1].
$$
Postponing the extremal cases $\alpha=0$ and $\alpha=1$, assume
$\alpha \in (0,1)$. Set
$$
\Delta_i =\{i^2+1,\dots, (i+1)^2\}, \quad i \geq 0.
$$
For every positive integer $k$ there exists exactly one number
$i=i(k)$ such that $k \in \Delta_i$. Now we set $n_1=1$ and define
a sequence $(n_k)_{k \geq 1}$ recursively by
\begin{eqnarray*}
n_k & = & \left\{ \begin{array}{lll} n_{k-1}+i+1 & \textrm{if}
&k = i^2+1~\textrm{for some $i$} \\
n_{k-1}+1 & \textrm{if} & k \in \left\{i^2+2,i^2+ \lceil 2 i
\alpha \rceil \right\}~\textrm{for some $i$} \\
n_{k-1}+i+1 & & \textrm{otherwise}.
\end{array}\right.
\end{eqnarray*}
For any $i \geq 0$, set
$$
p^{(i)}(x) = \sum_{j=1}^{2^i} a_j \cos 2 \pi j x, \qquad
r^{(i)}(x) = f(x)-p^{(i)}(x) = \sum_{j=2^i+1}^\infty a_j \cos 2
\pi j x.
$$
We want to calculate
$$
\int_0^1 \left( \sum_{k=1}^N f\left(2^{n_k} x\right) \right)^2 dx
$$
asymptotically. There is an $i$ such that $N \in \Delta_i$, and
since $N-i^2 \leq (i+1)^2-i^2 = 2i+1 \leq 2 \sqrt{N} + 1$, we have
by Lemma \ref{lemmavar}
\begin{eqnarray} \label{e1}
\left\| \sum_{k=1}^N f\left(2^{n_k} x\right) \right\| - \left\|
\sum_{k=1}^{i^2} f\left(2^{n_k} x\right) \right\| \leq \left\|
\sum_{k=i^2+1}^{N} f\left(2^{n_k} x\right) \right\| \leq 2 \left(2
\sqrt{N} + 1 \right)^{1/2}.
\end{eqnarray}
Using Lemma \ref{lemmavar} again, we get
\begin{eqnarray*}
& & \left\| \sum_{k=1}^{i^2} f\left(2^{n_k} x\right) \right\| \\
& = &  \left\| \sum_{h=0}^{i-1} \left( \sum_{k \in \Delta_h}
p^{(h)}\left(2^{n_k} x\right)+ \sum_{k \in \Delta_h} r^{(h)}
\left(2^{n_k} x\right) \right)\right\|,
\end{eqnarray*}
and
\begin{eqnarray}
& & \left\| \sum_{k=1}^{i^2} f\left(2^{n_k} x\right) \right\| -
\left\| \sum_{h=0}^{i-1} \left( \sum_{k \in \Delta_h} p^{(h)}
\left(2^{n_k} x\right) \right) \right\| \nonumber\\
& \leq & \sum_{h=0}^{i-1} \left\| \left( \sum_{k \in \Delta_h}
r^{(h)}\left(2^{n_k} x\right) \right) \right\| \nonumber\\
& \leq & \sum_{h=0}^{i-1} \sqrt{2 |\Delta_h|} 2^{-h/2} \nonumber\\
& \leq & \sum_{h=0}^{i-1} \sqrt{2 (2h+1)} 2^{-h/2} \nonumber\\
& \ll & 1.  \label{e2}
\end{eqnarray}
Now we calculate
$$
\left\| \sum_{h=0}^{i-1} \left( \sum_{k \in \Delta_h}
p^{(h)}\left(2^{n_k} x\right) \right) \right\|.
$$
By the construction of the sequence $(n_k)_{k \geq 1}$, the
functions
\begin{equation}\label{ort}
\sum_{k \in \Delta_{h_1}} p^{(h_1)}\left(2^{n_k} x\right), \qquad
\sum_{k \in \Delta_{h_2}} p^{(h_2)}\left(2^{n_k} x\right)
\end{equation}
are orthogonal if $h_1 \neq h_2$. 
In fact, if $h_2 > h_1$, and $k_1 \in \Delta_{h_1},~ k_2 \in
\Delta_{h_2}$, then $n_{k_2} \geq n_{k_1}+h_2+1$, which implies
that the largest frequency of a trigonometric function in the
Fourier series of $p^{(h_1)}(2^{n_{k_1}}x)$ is $2^{h_1}
2^{n_{k_1}} < 2^{n_{k_2}}$. Thus the functions in (\ref{ort}) are
really orthogonal. A similar argument shows that for fixed $h$ and
$k_1,k_2 \in \Delta_h$ the functions $p^{(h)}(2^{n_{k_1}}x)$ and
$p^{(h)}(2^{n_{k_2}}x)$ are orthogonal if not both $k_1$ and $k_2$
are in the set $\left\{h^2+1,h^2+ \lceil 2 h \alpha \rceil
\right\}$. Thus
\begin{eqnarray}
& & \int_0^1 \left( \sum_{h=0}^{i-1} \left( \sum_{k \in \Delta_h}
p^{(h)}\left(2^{n_k} x\right) \right) \right)^2 \label{g12} \\ & =
& \sum_{h=0}^{i-1} \left(  \int_0^1 \left( \sum_{k \in
\left\{h^2+1,h^2+ \lceil 2 h \alpha \rceil \right\}}
p^{(h)}\left(2^{n_k} x\right)  \right)^2 dx + \sum_{k \in
\left\{h^2+ \lceil 2 h \alpha \rceil +1, (h+1)^2 \right\}}
\|p^{(h)}\|^2 \right). \nonumber
\end{eqnarray}
For $h \to \infty$,
$$
\int_0^1 \left( \sum_{k \in \left\{h^2+1,h^2+ \lceil 2 h \alpha
\rceil \right\}} p^{(h)}\left(2^{n_k} x\right)  \right)^2 dx \sim
\gamma_f^2 \left( h^2 + \lceil 2 h \alpha \rceil - (h^2+1) \right)
\sim \gamma_f^2 2 h \alpha,
$$
and
$$
\sum_{k \in \left\{h^2+ \lceil 2 h \alpha \rceil +1, (h+1)^2
\right\}} \|p^{(h)}\|^2 \sim \|f\|^2 \left((h+1)^2 -(h^2+ \lceil 2
h \alpha \rceil)\right) \sim \|f\|^2 2h (1-\alpha).
$$
Thus by (\ref{g12}) for $i \to \infty$
\begin{eqnarray}
& & \int_0^1 \left( \sum_{h=0}^{i-1} \left( \sum_{k \in \Delta_h}
p^{(h)}\left(2^{n_k} x\right) \right) \right)^2 \nonumber\\
& \sim & \sum_{h=0}^{i-1} \left( \gamma_f^2 2h \alpha+\|f\|^2 2h
(1-\alpha)\right) \nonumber\\
& \sim & \left(\gamma_f^2 \alpha + \|f\|^2(1-\alpha) \right)
\sum_{h=0}^{i-1} 2h \nonumber\\
& \sim & \rho^2 i^2. \label{e4}
\end{eqnarray}
Combining (\ref{e1}), (\ref{e2}), (\ref{e4}) we finally obtain
\begin{equation}\label{permint2}
\int_0^1 \left( \sum_{k=1}^N f\left(2^{n_k} x\right) \right)^2 dx
\sim \rho^2 N \qquad \text{as} \ N\to\infty.
\end{equation}
Note that in our argument we assumed $\alpha\in (0, 1)$, i.e.\
that $\rho$ is an inner point of the interval $[\|f\|, \gamma_f]$.
The case $\alpha=1$ (i.e.\ $\rho=\gamma_f$) is trivial, with
$n_k=k$. In the case $\alpha=0$ we choose $(n_k)$ growing very
rapidly and the theory of lacunary series implies (\ref{permint2})
with $\rho=\|f\|$.

Relation (\ref{permint2}) is not identical with (\ref{permint}),
since the sequence $(n_k)$ is not a permutation of $\mathbb N$.
However, from $(n_k)$ we can easily construct a permutation
$\sigma$ such that (\ref{permint}) holds.
Let $H$ denote the set of positive integers not contained in
$(n_k)$ and insert the elements of $H$ into the sequence $n_1,
n_2, \ldots$ by leaving very rapidly increasing gaps between them.
The so obtained sequence  is a permutation $\sigma$ of $\mathbb N$
and if the gaps between the inserted elements grow sufficiently
rapidly, then clearly the asyimptotics of the integrals in
(\ref{permint}) and (\ref{permint2}) are the same, i.e.
(\ref{permint}) holds. This completes the proof of the fact that
the class of limits
$$
\lim_{N\to\infty} \frac{1}{N}\int_0^1 \left( \sum_{k=1}^N
f\left(2^{\sigma(k)} x\right) \right)^2 dx
$$
is identical with the interval $[\|f\|^2, \gamma_f^2]$.

In conclusion we show that without assuming the nonnegativity of
the Fourier coefficients of $f$, the class $\Gamma_f$ of limiting
variances in Theorem \ref{th1} for permuted sequences
$f(n_{\sigma(k)}x)$ is not necessarily the closed interval with
endpoints $\|f\|^2$ and $\gamma_f^2$.  Let
$$f(x)=\cos 2\pi x-\cos 4\pi x+\cos 8\pi x$$
and again $n_k=2^k$. Then taking into account the cancellations in
the sum $\sum _{k=1}^N f(n_kx)$ we get
$$
\sum_{k=1}^N f(n_kx) = \cos 4\pi x + \cos 16\pi x+ \cos 32\pi x
+\cos 64\pi x +\cos 128\pi x +\ldots
$$
whence
$$\int_0^1 \left(\sum_{k=1}^N f(n_kx)\right)^2 dx\sim N/2$$
so that $\gamma_f^2=1/2$ and clearly $\|f\|^2=3/2$. (Note that in
this case $\gamma_f< \|f\|$.) Now
\begin{align*}
&\sum_{k=1}^N f(4^kx)= \cos 8\pi x- \cos 16\pi x\\
&\phantom {9999999999999999} +2\cos 32\pi x-\cos 64 \pi x +2\cos
128\pi x-\cos 256x +2\cos 512 x - \ldots
\end{align*}
and thus
$$\int_0^1 \left(\sum_{k=1}^N f(4^kx)\right)^2 dx\sim 5N/2.$$
Similarly as above, we can get a permutation $\sigma$ of $\mathbb
N$ such that
$$\int_0^1 \left(\sum_{k=1}^N f(2^{\sigma(k)}x)\right)^2 dx\sim 5N/2.$$

\end{document}